\documentclass[11pt,article]{amsart}
\usepackage{geometry}

\usepackage{amsfonts,amsmath,amsthm,amssymb}

\usepackage{graphicx}
\usepackage{setspace}

\usepackage{dsfont}
\usepackage{grffile}

\usepackage [english]{babel}
\usepackage [autostyle, english = american]{csquotes}
\MakeOuterQuote{"}

\numberwithin{equation}{section}

\usepackage{hyperref}
\hypersetup{
    colorlinks = true,
    linkcolor = {blue}, urlcolor={blue}, citecolor={blue}
}

\newtheorem{thm}{Theorem}
\newtheorem{cor}{Corollary}
\newtheorem{lemma}{Lemma}

\numberwithin{equation}{section}
\numberwithin{lemma}{section}
\numberwithin{prop}{section}

\begin{document}

{
 \title[Correlations of multiplicative functions with their partial sums]{Correlations of multiplicative functions with their partial sums}

  \author{Gordon Chavez}

\date{}
  \maketitle
}

\begin{abstract}
Let $\zeta(.)$ denote the Riemann zeta function and let $a(.)$ and $A(.)$ respectively denote a multiplicative function and its corresponding summatory function. We consider the correlation 
$$
\langle a(n)A(n-1) \rangle (T) = \frac{1}{\zeta(1+\delta(T))}\sum_{n\leq T^{1-c}}\frac{a(n)A(n-1)}{n^{1+\delta(T)}}
$$
where $0<c<1$ is arbitrary and $0<\delta(T)=O\left(T^{c-1}\right)$ is suitably chosen. Let $\mu(.)$ and $\lambda(.)$ denote the M{\"o}bius function and the Liouville function respectively while $M(.)$ and $L(.)$ denote their corresponding summatory functions. Under the Riemann hypothesis and simplicity of the nontrivial zeros $\rho=1/2+ i \gamma$ of $\zeta(s)$ we show that 
$$
\langle \mu(n)M(n-1) \rangle (T)= -\frac{3}{\pi^{2}}\left(1-T^{(c-1)\delta(T)}\right)+\sum_{0<\gamma<T}\frac{1}{\left|\rho\zeta'(\rho)\right|^{2}}
$$
and
$$
\langle \lambda(n)L(n-1) \rangle (T)=\frac{1}{2}\left(\frac{1}{\zeta^{2}(1/2)}-1+T^{(c-1)\delta(T)}\right)+\sum_{0<\gamma<T}\left|\frac{\zeta(2\rho)}{\rho\zeta'(\rho)}\right|^{2}
$$
as $T\rightarrow \infty$ where  $0\leq T^{(c-1)\delta(T)}<1$. These results combined with numerical observations suggest that there is anticorrelation between $\mu(n)$ and $M(n-1)$ as well as between $\lambda(n)$ and $L(n-1)$, where the correlation is computed using a logarithmic average. This would imply effective upper bounds on $\left|1/\zeta'(\rho)\right|$.
\end{abstract}

\setcounter{tocdepth}{1}

\section{Introduction}
The M{\"o}bius function $\mu(n)$ is defined for $n \in \mathbb{N}$ as
$$
\mu(n)=\begin{cases} 
1; & \textnormal{$n=1$} \\
0; & \textnormal{$n$ has a repeated prime factor (not square-free)} \\ (-1)^{\omega(n)}; & \textnormal{$n$ has $\omega(n)$ distinct prime factors (square-free)} \end{cases} \label{mobius}
$$
It has the important property that its sum over the divisors of any $n>1$ vanishes, i.e., $\sum_{m|n}\mu(m)=0$
for all $n>1$. The summatory M{\"o}bius function $M(n)$, also called the Mertens function, is defined as the sum $M(n)=\sum_{m \leq n}\mu(m)=\sum_{\substack{m \nmid n \\ m<n}}\mu(m).$ The straightforward vanishing of the divisor sum $\sum_{m|n}\mu(m)$ contrasts strongly with the behavior of the non-divisor sum $M(n)$, as the latter grows in average magnitude and fluctuates irregularly as $n$ increases due to the pseudorandomness properties of $\mu(n)$, which are an important topic in contemporary number theory \cite{chowla} \cite{darty} \cite{mauduit} \cite{sarnak} \cite{bsz} \cite{greentao} \cite{green} \cite{bourgain} \cite{saw shu}. The growth rate of $M(n)$ has been linked to many important results in number theory. Landau \cite{landau} noted that the result $M(n)= o(n)$ is equivalent to the prime number theorem (PNT) and Littlewood showed \cite{littlewood} that the stronger and unproven statement 
\begin{equation}
M(n)=O\left(n^{1/2+\varepsilon}\right) \label{mertens rh}
\end{equation} 
for all $\varepsilon>0$ is equivalent to RH. The $n^{\varepsilon}$ term in (\ref{mertens rh}) has been sharpened significantly in recent years \cite{landaumob} \cite{titchmarshmob} \cite{montmob} \cite{soundmob} while the $n^{1/2}$ term persists. This supports the popular heuristic that $M(n)$ behaves approximately like a random walk. Let $\rho$ denote the nontrivial zeros of the Riemann zeta function $\zeta(s)$, which, under the Riemann hypothesis (RH), can be written $\rho=1/2+ i\gamma$. Additionally note that the "simple zeros conjecture" (SZC) is the conjecture that all $\rho$ are simple. The behavior of $M(n)$ is known to be linked with the values of $1/\zeta'(\rho)$ chiefly through the fact that 
\begin{equation}
M(n)=\lim_{T\rightarrow \infty}\sum_{|\gamma|<T}\frac{n^{\rho}}{\rho\zeta'(\rho)}+O(1) \label{da fact}
\end{equation}
under RH and SZC.

The Liouville function $\lambda(n)$ is closely related to $\mu(n)$ but displays some key differences. It is defined for $n\in \mathbb{N}$ as
$$
\mu(n)=\begin{cases} 
1; & \textnormal{$n=1$} \\
(-1)^{\Omega(n)}; & \textnormal{$n$ has $\Omega(n)$ prime factors (counting multiplicities)} \end{cases} \label{liouville}
$$
The Liouville function also has a well-understood sum over the divisors of any $n>1$: $\sum_{m|n}\lambda(m)$ is equal to 1 for $n$ a perfect square and is equal to $0$ otherwise. The summatory Liouville function $L(n)$ is defined as the sum $L(n)=\sum_{m\leq n}\lambda(n).$
Similarly to $M(n)$, $L(n)$ increases in average magnitude and fluctuates irregularly as $n$ increases, again due to the pseudorandomness properties of $\lambda(n)$, which are also an important topic of contemporary research \cite{mrt} \cite{mr} \cite{tao} \cite{tt}. Very similarly to $M(n)$'s (\ref{mertens rh}), the unproven statement
\begin{equation}
L(n)=O\left(n^{1/2+\varepsilon}\right) \label{suml rh}
\end{equation}
for all $\varepsilon>0$ is equivalent to RH \cite{humphries}. Under RH and SZC $L(n)$ is linked with $1/\zeta'(\rho)$ through the formula \cite{fawaz} \cite{humphries}
\begin{equation}
L(n)=\frac{n^{1/2}}{\zeta(1/2)}+\lim_{T\rightarrow \infty}\sum_{|\gamma|<T}\frac{\zeta(2\rho)n^{\rho}}{\rho\zeta'(\rho)}+O(1),
\label{suml fact}
\end{equation}
which bears clear similarities to $M(n)$'s (\ref{da fact}) along with several differences. Most notably, the leading term combined with the fact that $\zeta(1/2)<0$ gives a negative bias to $L(n)$.

\subsection{Main results}
Our first result describes the correlation between $\mu(n)$ and $M(n-1)$:
\begin{thm}
Under RH and SZC, for any $0<c<1$ and suitably chosen $0<\delta(T)=O\left(T^{c-1}\right)$, 
\begin{equation}
\frac{1}{\zeta\left(1+\delta(T)\right)}\sum_{n\leq T^{1-c}}\frac{\mu(n)M(n-1)}{n^{1+\delta(T)}}=-\frac{3}{\pi^{2}}\left(1-T^{(c-1)\delta(T)}\right)+\sum_{0<\gamma<T}\frac{1}{\left|\rho\zeta'(\rho)\right|^{2}} \label{main lim 1}
\end{equation}
as $T\rightarrow \infty$ where $0\leq T^{(c-1)\delta(T)}<1$.
\label{mainmainthm}
\end{thm}
\hspace{-.4cm}The next result describes the correlation between $\lambda(n)$ and $L(n-1)$:
\begin{thm}
Under RH and SZC, for any $0<c<1$ and suitably chosen $0<\delta(T)=O\left(T^{c-1}\right)$, 
\begin{equation}
\frac{1}{\zeta\left(1+\delta(T)\right)}\sum_{n\leq T^{1-c}}\frac{\lambda(n)L(n-1)}{n^{1+\delta(T)}}=\frac{1}{2}\left(\frac{1}{\zeta^{2}(1/2)}-1+T^{(c-1)\delta(T)}\right)+\sum_{0<\gamma<T}\left|\frac{\zeta(2\rho)}{\rho\zeta'(\rho)}\right|^{2} \label{main lim 2}
\end{equation}
as $T\rightarrow \infty$ where $0\leq T^{(c-1)\delta(T)}<1$.
\label{mainmainthm2}
\end{thm}
\hspace{-.4cm}Note that $\frac{1}{2}\left(\frac{1}{\zeta^{2}(1/2)}-1\right)\approx -.2655$.

\subsection{Numerical results}
We present numerical evidence in Figure \ref{fig1}. In Figure \ref{fig1}'s left-hand plot we depict the first ten million values of 
\begin{equation}
\frac{1}{\log N}\sum_{n\leq N}\frac{\mu(n)M(n-1)}{n}. \label{log avg norm}
\end{equation}
along with the value $-\frac{3}{\pi^{2}}+\sum_{0<\gamma\leq\gamma_{n}}\frac{1}{\left|\rho\zeta'(\rho)\right|^{2}}$ with $n=1,000,000$, using values of $\zeta'(\rho)$ calculated by Hughes, Martin, \& Pearce-Crump \cite{hmpc}. This sum over the first one million $\gamma$ has the approximate value 
\begin{equation}
\sum_{0<\gamma\leq\gamma_{1,000,000}}\frac{1}{\left|\rho\zeta'(\rho)\right|^{2}} \approx .0145
\end{equation}
It is likely that taking the sum over a larger set of $\gamma$ would produce a sharper estimate for (\ref{log avg norm}). These numerical results indicate that the result (\ref{main lim 1}) approximates the values of (\ref{log avg norm}).

\begin{figure}
\centering
\includegraphics[width=\linewidth]{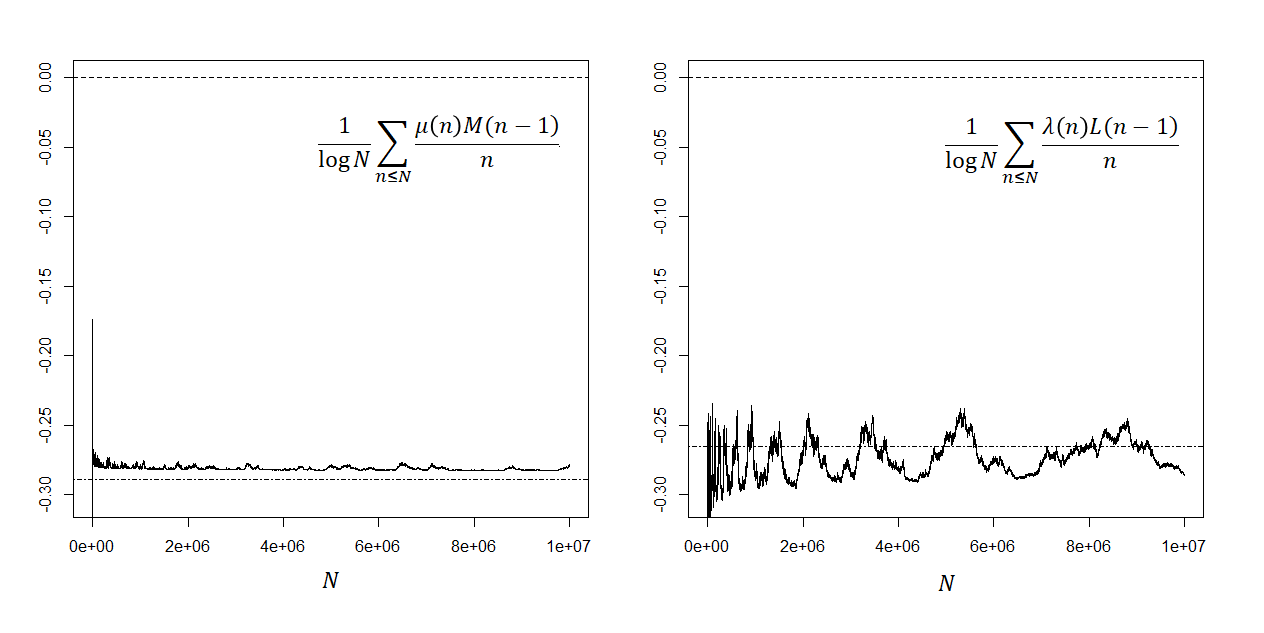}
\caption{Plots of (\ref{log avg norm}) [Left] and (\ref{log avg norm ll}) [Right] for $1\leq N \leq 10^{7}$ with horizontal lines (Dashed) at zero. In the left-hand plot there is additionally a line (Dot-Dashed) at $-\frac{3}{\pi^{2}}+\sum_{0<\gamma\leq \gamma_{n}}\frac{1}{\left|\rho\zeta'\left(\rho\right)\right|^{2}}$ with $n=1,000,000$. In the right-hand plot there is additionally a line (Dot-Dashed) at $\frac{1}{2}\left(\frac{1}{\zeta^{2}(1/2)}-1\right)$.}
\label{fig1}
\end{figure}

In Figure \ref{fig1}'s right-hand plot we depict the first ten million values of 
\begin{equation}
\frac{1}{\log N}\sum_{n\leq N}\frac{\lambda(n)L(n-1)}{n}, \label{log avg norm ll}
\end{equation}
along with the value $\frac{1}{2}\left(\frac{1}{\zeta^{2}(1/2)}-1\right)$. It is clear that this leading term in (\ref{main lim 2}) is a good approximation for the values of (\ref{log avg norm ll}). We will show below that the term $\frac{1}{2\zeta^{2}(1/2)}$, which has the net effect of increasing the correlation (\ref{main lim 2}), comes from the leading term $\frac{n^{1/2}}{\zeta(1/2)}$ in (\ref{suml fact}) that produces a bias in $L(n)$.

In Figure \ref{fig1} both (\ref{log avg norm}) and (\ref{log avg norm ll}) appear to be robustly negative. We additionally note that Theorem \ref{mainmainthm} has the immediate corollary
\begin{cor}
If (\ref{main lim 1}) holds and its right-hand side is negative, then 
\begin{equation}
\frac{1}{\zeta'(\rho)}=o\left(\left|\rho\right|\right) \label{base ub}
\end{equation}
as $\left|\gamma\right|\rightarrow \infty$.
\label{maincor}
\end{cor}
\hspace{-.4cm}while Theorem \ref{mainmainthm2} has the corollary
\begin{cor}
If (\ref{main lim 2}) holds and its right-hand side is negative, then under RH 
\begin{equation}
\frac{1}{\zeta'(\rho)}=o\left(\left|\rho\right|\log\log \left|\gamma\right|\right) \label{base ub 2}
\end{equation}
as $\left|\gamma\right|\rightarrow \infty$.
\label{maincor2}
\end{cor}
\hspace{-.4cm}These corollaries combined with the numerical results in Figure \ref{fig1} notably suggest that there are effective upper bounds (\ref{base ub})-(\ref{base ub 2}) on $\left|1/\zeta'(\rho)\right|$. Finding effective upper bounds on $1/\left|\zeta'(\rho)\right|$ is a longstanding open problem. Little is known \cite{gonek1} \cite{hejhal} \cite{hko} \cite{bui gonek mili} about such upper bounds despite a great deal of progress on lower bounds of $1/\left|\zeta'(\rho)\right|$ \cite{gonek1} \cite{mandng1} \cite{titchmarsh} \cite{hlz} \cite{rudsound} \cite{gao zhao}. These numerical results must be interpreted carefully, however, as such results may be misleading. A famous example of this is Merten's assertion, based on numerical evidence, that $\left|M(N)\right|<\sqrt{N}$ for all $N$, which was proven false \cite{odlyzko} despite no counterexample being observed. 


We will next present the proof of Theorem \ref{mainmainthm} and Corollary \ref{maincor} followed by the proof of Theorem \ref{mainmainthm2} and Corollary \ref{maincor2}. These proofs use many of the same ideas and methods.


\section{Proofs}

\subsection{Proof outline(s)}
We first give a quick outline of the proofs of Theorems \ref{mainmainthm} and \ref{mainmainthm2}: For Theorem \ref{mainmainthm} we consider 
\begin{equation}
\frac{1}{\zeta(1+\delta)}\sum_{n\leq N}\frac{\mu(n)M(n-1)}{n^{1+\delta}} \label{sum of interest}
\end{equation}
with $\delta>0$, while for Theorem \ref{mainmainthm2} we consider 
\begin{equation}
\frac{1}{\zeta(1+\delta)}\sum_{n\leq N}\frac{\lambda(n)L(n-1)}{n^{1+\delta}} \label{sum of interest2}
\end{equation}
also with $\delta>0$. We write $M(n-1)$ and $L(n-1)$ using versions of (\ref{da fact}) and (\ref{suml fact}) for sums over $\rho$ with $\left|\gamma\right|<T$. We then evaluate the sum in (\ref{sum of interest}) by using versions of the elementary identities 
\begin{equation}
\sum_{n=1}^{\infty}\frac{\mu(n)}{n^{s}}=\frac{1}{\zeta(s)} \label{elem}
\end{equation}
and
\begin{equation}
\sum_{n=1}^{\infty}\frac{\mu^{2}(n)}{n^{s}}=\frac{\zeta(s)}{\zeta(2s)} \label{dlmf gen}
\end{equation}
for $n\leq N$. To evaluate the sum in (\ref{sum of interest2}) we use versions of the identities 
\begin{equation}
\sum_{n=1}^{\infty}\frac{\lambda(n)}{n^{s}}=\frac{\zeta(2s)}{\zeta(s)} \label{elem2}
\end{equation}
and
\begin{equation}
\sum_{n=1}^{\infty}\frac{\lambda^{2}(n)}{n^{s}}=\sum_{n=1}^{\infty}\frac{1}{n^{s}}=\zeta(s). \label{dlmf gen2}
\end{equation}
for $n\leq N$. For both theorems we then apply the Laurent expansion for $\zeta(1+\delta)$ along with a few other facts to give expressions for (\ref{sum of interest}) and (\ref{sum of interest2}). We then choose $N=T^{1-c}$ for some $0<c<1$ for both cases. Lastly, for each theorem, we make an especially careful choice of $\delta=\delta(T)$ such that $\delta(T) \rightarrow 0$ sufficiently quickly as $T\rightarrow \infty$. Taking $T\rightarrow \infty$ then gives the results  (\ref{main lim 1}) and (\ref{main lim 2}). 

\subsection{Proof of Theorem \ref{mainmainthm} and Corollary \ref{maincor}}
\label{sec2}
We first present three lemmas that describe the aforementioned formula for $M(n-1)$, formulae for finite versions of (\ref{elem}) and (\ref{dlmf gen}), and our choice of $\delta(T)$ respectively:
\subsubsection{Lemmas}
\begin{lemma}
Under RH and SZC 
\begin{equation}
M(n-1)=\sum_{\left|\gamma\right|<T}\frac{n^{\rho}}{\rho\zeta'(\rho)}-2-\sum_{k=1}^{\infty}\frac{\left(\frac{2\pi i}{n}\right)^{2k}}{(2k)!k\zeta(2k+1)}-\frac{\mu(n)}{2}+O\left(\frac{n^{2}}{T^{1-\epsilon}}\right) \label{mertens big}
\end{equation}
where $\epsilon>0$ is arbitrary. \label{titchlemma}
\end{lemma}
\begin{proof}
From Perron's formula (see e.g. \cite{titchmarsh} 3.12.2) we may write 
\begin{equation}
M(n)=\frac{1}{2\pi i}\int_{2-i\infty}^{2+i\infty}\frac{x^{s}}{s\zeta(s)}ds+O\left(\frac{n^{2}}{T}\right)+O\left(\frac{\log n}{T}\right)+O\left(\frac{1}{nT}\right). \label{perry}
\end{equation}
Titchmarsh \& Heath-Brown \cite{titchmarsh} (Thm. 14.27) then showed that under SZC one may apply the calculus of residues to give
\begin{align}
\frac{1}{2\pi i}\int_{2-i\infty}^{2+i\infty}\frac{x^{s}}{s\zeta(s)}ds=\sum_{\left|\gamma\right|<T}\frac{n^{\rho}}{\rho\zeta'(\rho)}-2-\sum_{k=1}^{\infty}\frac{\left(\frac{2\pi i}{n}\right)^{2k}}{(2k)!k\zeta(2k+1)}\nonumber\\
+O\left(\frac{1}{T}\int_{2}^{\infty}\left(\frac{2\pi}{n}\right)^{u}e^{u-(u-1/2)\log u}du\right)+\frac{1}{2\pi i}\int_{-1+iT}^{2+iT}\frac{n^{s}}{s\zeta(s)}ds. \label{titch1}
\end{align}
The first integral on (\ref{titch1})'s second line is $O\left(\frac{1}{n^{2}T}\right)$ and
\begin{equation}
\frac{1}{2\pi i}\int_{-1+iT}^{2+iT}\frac{n^{s}}{s\zeta(s)}ds=O\left(n^{2}\int_{-1+iT}^{2+iT}\left|\frac{1}{s\zeta(s)}\right|ds\right). \label{titch2}
\end{equation}
Under RH one then may choose a $T=T_{\nu}$ such that $\nu \leq T_{\nu}\leq \nu+1$ and 
\begin{equation}
\frac{1}{\zeta(s)}=O\left(t^{\epsilon}\right) \label{titch3}
\end{equation}
for $s=\sigma+iT_{\nu}$ with $1/2\leq \sigma \leq 2$ and arbitrary $\epsilon>0$ (\cite{titchmarsh} 14.16.2). Thus
\begin{equation}
\frac{1}{\zeta(s)}=O\left(\frac{\left|t\right|^{\sigma-1/2}}{\left|\zeta(1-s)\right|}\right)=O\left(t^{\epsilon}\right) \label{titch4}
\end{equation}
for $s=\sigma+iT_{\nu}$ with $-1\leq \sigma \leq 1/2$. Hence 
\begin{equation}
\int_{-1+iT_{\nu}}^{2+iT_{\nu}}\left|\frac{1}{s\zeta(s)}\right|ds=O\left(\frac{1}{T_{\nu}^{1-\epsilon}}\right). \label{titch5}
\end{equation}
Combining (\ref{titch1}), (\ref{titch2}), and (\ref{titch5}) with (\ref{perry}) then shows that 
\begin{equation}
M(n)=\sum_{\left|\gamma\right|<T}\frac{n^{\rho}}{\rho\zeta'(\rho)}-2-\sum_{k=1}^{\infty}\frac{\left(\frac{2\pi i}{n}\right)^{2k}}{(2k)!k\zeta(2k+1)}+\frac{\mu(n)}{2}+O\left(\frac{n^{2}}{T^{1-\epsilon}}\right) \label{pre mertens big}
\end{equation}
Since $M(n-1)=M(n)-\mu(n)$ we then have (\ref{mertens big}).
\end{proof}

\begin{lemma}
Under RH, for any $\varepsilon>0$ there exists an $N_{1}(\varepsilon)<\infty$ such that for all $N\geq N_{1}(\varepsilon)$ 
\begin{align}
\sum_{n\leq N}\frac{\mu(n)}{n^{s}}=\frac{1}{\zeta(s)}+\textnormal{Err}_{1}(s,N) \label{elem finite}
\end{align}
with 
\begin{align}
\left|\textnormal{Err}_{1}(s,N)\right|=\left(1+\left|\frac{s}{1/2+\varepsilon-s}\right|\right)O\left(\frac{1}{N^{s-1/2-\varepsilon}}\right) \label{elem err}
\end{align}
for $\textnormal{Re}(s)>1/2$ where the implied constant in (\ref{elem err})'s $O(.)$ term has no dependence on $s$. Additionally for all $N\geq N_{2}$ (unconditionally)
\begin{align}
\sum_{n\leq N}\frac{\mu^{2}(n)}{n^{s}}=\frac{\zeta(s)}{\zeta(2s)}+\textnormal{Err}_{2}(s,N) \label{dlmf gen finite}
\end{align}
with
\begin{align}
\textnormal{Err}_{2}(s,N)=-\frac{6}{\pi^{2}}\frac{1}{s-1}\frac{1}{N^{s-1}}+\left(1-\frac{s}{s-1/2}\right)O\left(\frac{1}{N^{s-1/2}}\right) \label{dlmf gen err}
\end{align}
for $\textnormal{Re}(s)>1$ where $N_{2}<\infty$ and the implied constant in (\ref{dlmf gen err})'s $O(.)$ term similarly has no dependence on $s$. For $s\in \mathbb{R}$ with $s>1$ we also have 
\begin{equation}
0\leq \left|\textnormal{Err}_{2}(s,N)\right|<\frac{\zeta(s)}{\zeta(2s)}. \label{dlmf gen err 2}
\end{equation} \label{elem lemma}
\end{lemma}
\begin{proof}
We first note from (\ref{mertens rh}) that, under RH, for any $\varepsilon>0$ there exists an $N_{1}(\varepsilon)<\infty$ such that for all $N\geq N_{1}(\varepsilon)$
\begin{equation}
\left|M(N)\right|\leq C_{1}(\varepsilon) N^{1/2+\varepsilon} \label{mertens rh useful}
\end{equation}
where $0<C_{1}(\varepsilon)<\infty$. We thus choose some $\varepsilon>0$ and apply partial summation with (\ref{mertens rh useful}) to show that for $N\geq N_{1}(\varepsilon)$ (\ref{elem}) has the finite series version
\begin{align}
\sum_{n\leq N}\frac{\mu(n)}{n^{s}}=\frac{1}{\zeta(s)}-\sum_{n>N}^{\infty}\frac{\mu(n)}{n^{s}}=\frac{1}{\zeta(s)}+\frac{M(N)}{N^{s}}-s\int_{N}^{\infty}\frac{M(u)}{u^{s+1}}du \nonumber \\ =\frac{1}{\zeta(s)}+\textnormal{Err}_{1}(s,N) \label{pre elem finite}
\end{align}
with 
\begin{align}
\left|\textnormal{Err}_{1}(s,N)\right|\leq C_{1}(\varepsilon)\left(1+\left|\frac{s}{1/2+\varepsilon-s}\right|\right)\frac{1}{N^{s-1/2-\varepsilon}} \label{pre elem err}
\end{align}
for $\textnormal{Re}(s)>1/2$. Combining (\ref{pre elem finite}) and (\ref{pre elem err}) proves (\ref{elem finite})-(\ref{elem err}). To prove the next statements we note the distribution of square-free integers (see e.g. \cite{hardy wright}):
\begin{equation}
\sum_{n\leq N}\mu^{2}(n)=\frac{6}{\pi^{2}}N+R(N) \label{sq free}
\end{equation}
where
\begin{equation}
\left|R(N)\right|\leq C_{2}N^{1/2} \label{sq free 2}
\end{equation}
for all $N\geq N_{2}$ where $N_{2}$ and $C_{2}$ are finite. We apply partial summation to write 
\begin{align}
\sum_{n\leq N}\frac{\mu^{2}(n)}{n^{s}}=\frac{\zeta(s)}{\zeta(2s)}-\sum_{n>N}^{\infty}\frac{\mu^{2}(n)}{n^{s}} \nonumber \\=\frac{\zeta(s)}{\zeta(2s)}+\frac{\sum_{n\leq N}\mu^{2}(n)}{N^{s}}-s\int_{N}^{\infty}\frac{\sum_{n\leq u}\mu^{2}(n)}{u^{s+1}}du \nonumber \\ =\frac{\zeta(s)}{\zeta(2s)}+\textnormal{Err}_{2}(s,N). \label{pre dlmf gen finite}
\end{align}
Setting $N\geq N_{2}$ and applying (\ref{sq free})-(\ref{sq free 2}) in (\ref{pre dlmf gen finite})'s second line then proves (\ref{dlmf gen finite})-(\ref{dlmf gen err}). The bounds (\ref{dlmf gen err 2}) are clear from the fact that $\sum_{n}\frac{\mu^{2}(n)}{n^{s}}$ is real, strictly positive, and convergent for $s\in \mathbb{R}$ with $s>1$. 
\end{proof}
Note that (\ref{elem err})'s $\varepsilon$  is distinct from (\ref{mertens big})'s $\epsilon$. 

\begin{lemma}
Let
\begin{equation}
X(T)=\frac{1}{\underset{\gamma<T}{\textnormal{max}}\left(\underset{k\rightarrow \infty}{\lim\sup}\left|\frac{\overline{\zeta^{(k+2)}}(\rho)}{(k+2)!}\right|^{1/k}\right)},  \label{x bound}
\end{equation}
\begin{align}
S(T)=\textnormal{max}\left\{\sum_{0<\gamma<T}\frac{1}{\left|\rho\zeta'(\rho)\right|^{2}},\sup_{0\leq x \leq \frac{X(T)}{2}}\left\{\left|\sum_{\left|\gamma\right|<T}\frac{x\sum_{k=2}^{\infty}\overline{\zeta^{(k)}}(\rho)\frac{x^{k-2}}{k!}}{\rho\left|\zeta'(\rho)\right|^{2}\zeta\left(\bar{\rho}+x\right)}\right|, \right. \right. \nonumber \\ \left. \left. \left|\sum_{\left|\gamma\right|<T}\frac{1}{\rho\zeta'(\rho)}\left|\frac{1+x-\rho}{\varepsilon-1/2-x+\rho}\right|\right|\right\} \right\},
\label{da sup}
\end{align}
and
\begin{equation}
\delta(T)=\frac{1}{N}\textnormal{min}\left\{\frac{1}{\textnormal{max}\left\{S(T),1\right\}},X(T)\right\} \label{sneaky}
\end{equation}
with $\varepsilon>0$. Then under RH and SZC (\ref{sneaky})'s $\delta(T)$ is greater than 0 for all $T<\infty$. Additionally 
\begin{equation}
O\left(\delta(T)\right)\sum_{0<\gamma<T}\frac{1}{\left|\rho\zeta'(\rho)\right|^{2}}=O\left(\frac{1}{N}\right), \label{probterm1}
\end{equation}
\begin{equation}
O\left(\delta(T)\right)\sum_{\left|\gamma\right|<T}\frac{\delta(T)\sum_{k=2}^{\infty}\overline{\zeta^{(k)}}(\rho)\frac{\left(\delta(T)\right)^{k-2}}{k!}}{\rho \left|\zeta'(\rho)\right|^{2}\zeta\left(1+\delta(T)-\rho\right)}=O\left(\frac{1}{N}\right), \label{probterm2}
\end{equation}
\begin{equation}
O\left(\frac{\delta(T)}{N^{\delta(T)-\varepsilon}}\right)\sum_{\left|\gamma\right|<T}\frac{1}{\rho\zeta'(\rho)}\left|\frac{1+\delta(T)-\rho}{\varepsilon-1/2-\delta(T)+\rho}\right|=O\left(\frac{1}{N^{1+\delta(T)-\varepsilon}}\right), \label{probterm3}
\end{equation}
and 
\begin{equation}
\delta(T)=O\left(\frac{1}{N}\right) \label{sneaked}
\end{equation}
as $N\rightarrow \infty$. 
\label{delta lemma}
\end{lemma}
\begin{proof}
We quickly note that under RH
\begin{align}
\sum_{\left|\gamma\right|<T}\frac{1}{\rho\left|\zeta'(\rho)\right|^{2}}=\sum_{0<\gamma<T}\left(\frac{1}{1/2+i\gamma}+\frac{1}{1/2-i\gamma}\right)\frac{1}{\left|\zeta'(\rho)\right|^{2}} \nonumber \\ =\sum_{0<\gamma<T}\left(\frac{1/2-i\gamma}{1/4+\gamma^{2}}+\frac{1/2+i\gamma}{1/4+\gamma^{2}}\right)\frac{1}{\left|\zeta'(\rho)\right|^{2}}=\sum_{0<\gamma<T}\frac{1}{\left|\rho \zeta'(\rho)\right|^{2}}.
\label{write it out}
\end{align}
We next note that under SZC $\zeta'(\rho)\neq 0$ for all $\rho$ and hence
\begin{equation}
\sum_{0<\gamma<T}\frac{1}{\left|\rho \zeta'(\rho)\right|^{2}}<\infty \label{easy sum}
\end{equation}
for all $T<\infty$. We note that, under RH, $1-\rho=\bar{\rho}$, and hence under RH and SZC
\begin{equation}
\frac{x}{\zeta(1-\rho+x)}=\frac{x}{\zeta\left(\bar{\rho}+x\right)}<\infty \label{phew bound}
\end{equation}
for all $x\geq 0$. We also note from the analyticity of $\zeta(s)$ that $\overline{\zeta^{(k)}}(\rho)<\infty$ for all $k$ and all $\rho$ and hence that
\begin{equation}
\left|\sum_{k=2}^{\infty}\overline{\zeta^{(k)}}(\rho)\frac{x^{k-2}}{k!}\right|=\left|\sum_{k=0}^{\infty}\overline{\zeta^{(k+2)}}(\rho)\frac{x^{k}}{(k+2)!}\right|<\infty \label{mini hard sum}
\end{equation}
for all 
\begin{equation}
x<\frac{1}{\underset{k\rightarrow \infty}{\lim\sup}\left|\frac{\overline{\zeta^{(k+2)}}(\rho)}{(k+2)!}\right|^{1/k}}. \label{mini x bound}
\end{equation}
Thus (\ref{mini hard sum}) holds for all $\rho$ with $\gamma<T$ if 
\begin{equation}
x<\frac{1}{\underset{\gamma<T}{\textnormal{max}}\left(\underset{k\rightarrow \infty}{\lim\sup}\left|\frac{\overline{\zeta^{(k+2)}}(\rho)}{(k+2)!}\right|^{1/k}\right)}=X(T).
\end{equation}
By the analyticity of $\zeta(s)$ we have
\begin{equation}
\underset{k\rightarrow \infty}{\lim \sup}\left|\frac{\overline{\zeta^{(k)}}(\rho)}{k!}\right|^{1/k}<\infty \label{by analyticity}
\end{equation}
and hence (\ref{mini x bound})'s right-hand side and (\ref{x bound})'s $X(T)$ are greater than $0$. From (\ref{phew bound})-(\ref{by analyticity}) we may thus conclude that 
\begin{equation}
\left|\sum_{\left|\gamma\right|<T}\frac{x\sum_{k=2}^{\infty}\overline{\zeta^{(k)}}(\rho)\frac{x^{k-2}}{k!}}{\rho \left|\zeta'(\rho)\right|^{2}\zeta\left(\bar{\rho}+x\right)}\right|<\infty \label{harder sum}
\end{equation}
for all $0\leq x<X(T)$ and $T<\infty$ under RH and SZC. Additionally it is clear that under SZC 
\begin{equation} 
\left|\sum_{\left|\gamma\right|<T}\frac{1}{\rho\zeta'(\rho)}\left|\frac{1+x-\rho}{\varepsilon-1/2-x+\rho}\right|\right|<\infty \label{clear sum}
\end{equation}
for all $\varepsilon>0$, $0\leq x<\infty$, and $T<\infty$. By (\ref{easy sum}), (\ref{harder sum}), and (\ref{clear sum}) we have (\ref{da sup})'s $S(T)<\infty$ for all $T<\infty$ and hence (\ref{sneaky})'s $\delta(T)>0$ for all $T<\infty$. The results (\ref{probterm1})-(\ref{sneaked}) are clear from (\ref{da sup})-(\ref{sneaky}).
\end{proof}

We next proceed with the proof of Theorem \ref{mainmainthm}:

\subsubsection{Proof of Theorem \ref{mainmainthm}}
\begin{proof}
We first apply Lemma \ref{titchlemma}'s result (\ref{mertens big}) to write
\begin{align}
\sum_{n\leq N}\frac{\mu(n)M(n-1)}{n^{1+\delta}}=\sum_{\left|\gamma\right|<T}\frac{1}{\rho\zeta'(\rho)}\sum_{n\leq N}\frac{\mu(n)}{n^{1-\rho+\delta}}-2\sum_{n\leq N}\frac{\mu(n)}{n^{1+\delta}}\nonumber \\ -\sum_{k=1}^{\infty}\frac{\left(2\pi i\right)^{2k}}{(2k)!k\zeta(2k+1)}\sum_{n\leq N}\frac{\mu(n)}{n^{2k+1+\delta}}-\frac{1}{2}\sum_{n\leq N}\frac{\mu^{2}(n)}{n^{1+\delta}}+O\left(\frac{1}{T^{1-\epsilon}}\sum_{n\leq N}n^{1-\textnormal{Re}(\delta)}\right). \label{piece 1 raw}
\end{align}
We then apply Lemma \ref{elem lemma}'s results (\ref{elem finite})-(\ref{elem err}) and (\ref{dlmf gen finite})-(\ref{dlmf gen err}) in (\ref{piece 1 raw}) to write
\begin{align}
\sum_{n\leq N}\frac{\mu(n)M(n-1)}{n^{1+\delta}}\nonumber \\=\sum_{\left|\gamma\right|<T}\frac{1}{\rho\zeta'(\rho)}\left(\frac{1}{\zeta(1-\rho+\delta)}+\left(1+\left|\frac{1+\delta-\rho}{\varepsilon-1/2-\delta+\rho}\right|\right)O\left(\frac{1}{N^{\textnormal{Re}(\delta)-\varepsilon}}\right)\right)\nonumber \\ -2\left(\frac{1}{\zeta(1+\delta)}+\left(1+\left|\frac{1+\delta}{\varepsilon-1/2-\delta}\right|\right)O\left(\frac{1}{N^{1/2+\textnormal{Re}(\delta)-\varepsilon}}\right)\right) \nonumber \\-\sum_{k=1}^{\infty}\frac{\left(2\pi i\right)^{2k}}{(2k)!k\zeta(2k+1)}\left(\frac{1}{\zeta(2k+1+\delta)} \right. \nonumber \\ \left. +\left(1+\left|\frac{2k+1+\delta}{\varepsilon-2k-1/2-\delta}\right|\right)O\left(\frac{1}{N^{5/2+\textnormal{Re}(\delta)-\varepsilon}}\right)\right) \nonumber \\-\frac{1}{2}\left(\frac{\zeta(1+\delta)}{\zeta(2+2\delta)}-\frac{6}{\pi^{2}}\frac{1}{\delta N^{\textnormal{Re}(\delta)}}+\left(1-\frac{1+\delta}{1/2+\delta}\right)O\left(\frac{1}{N^{1/2+\textnormal{Re}(\delta)}}\right)\right) \nonumber \\ +O\left(\frac{N^{2-\textnormal{Re}(\delta)}}{T^{1-\epsilon}}\right)\label{piece 1 done}
\end{align}
for $\textnormal{Re}(\delta)>0$ and $N\geq \textnormal{max}\left\{N_{1}(\varepsilon), N_{2}\right\}$ where we choose $0<\varepsilon<1/2$. Similarly to (\ref{elem err}) and (\ref{dlmf gen err}), there is no hidden dependence on $\delta$ in (\ref{piece 1 done})'s $O(.)$ terms. 

We next consider real $\delta>0$ and note from the Laurent expansion for $\zeta(s)$ that as $\delta \rightarrow 0$
\begin{equation}
\frac{1}{\zeta(1+\delta)}=\frac{1}{\frac{1}{\delta}+O(1)}=\delta+O\left(\delta^{2}\right). \label{laurent}
\end{equation}
Additionally, since $\zeta^{(k)}(\bar{s})=\overline{\zeta^{(k)}}(s)$ and again, under RH, $1-\rho=\bar{\rho}$, we have  
\begin{equation}
\zeta(1-\rho+\delta)=\zeta\left(\bar{\rho}+\delta\right)=\overline{\zeta'}(\rho)\delta+\sum_{k=2}^{\infty}\overline{\zeta^{(k)}}(\rho)\frac{\delta^{k}}{k!}. \label{zz sum help}
\end{equation}
From (\ref{laurent}), (\ref{zz sum help}), and the fact that 
\begin{equation}
\frac{1}{a+b}=\frac{1}{a}-\frac{b}{a^{2}+ab} \label{fractions!}
\end{equation}
we have 
\begin{align}
\frac{1}{\zeta(1+\delta)\zeta(1-\rho+\delta)}=\frac{\delta}{\zeta\left(\bar{\rho}+\delta\right)}\left(1+O\left(\delta\right)\right) \nonumber \\=\frac{\delta}{\overline{\zeta'}(\rho)\delta+\sum_{k=2}^{\infty}\overline{\zeta^{(k)}}(\rho)\frac{\delta^{k}}{k!}}\left(1+O(\delta)\right)=\frac{1}{\overline{\zeta'}(\rho)+\sum_{k=2}^{\infty}\overline{\zeta^{(k)}}(\rho)\frac{\delta^{k-1}}{k!}}\left(1+O(\delta)\right)\nonumber \\ 
=\left(\frac{1}{\overline{\zeta'}(\rho)}-\frac{\sum_{k=2}^{\infty}\overline{\zeta^{(k)}}(\rho)\frac{\delta^{k-1}}{k!}}{\overline{\zeta'}^{2}(\rho)+\overline{\zeta'}(\rho)\sum_{k=2}^{\infty}\overline{\zeta^{(k)}}(\rho)\frac{\delta^{k-1}}{k!}}\right)\left(1+O(\delta)\right)\nonumber \\
=\left(\frac{1}{\overline{\zeta'}(\rho)}-\frac{\delta\sum_{k=2}^{\infty}\overline{\zeta^{(k)}}(\rho)\frac{\delta^{k-1}}{k!}}{\overline{\zeta'}(\rho)\zeta\left(\bar{\rho}+\delta\right)}\right)\left(1+O(\delta)\right)\nonumber \\
=\left(\frac{1}{\overline{\zeta'}(\rho)}-\frac{\delta^{2}\sum_{k=2}^{\infty}\overline{\zeta^{(k)}}(\rho)\frac{\delta^{k-2}}{k!}}{\overline{\zeta'}(\rho)\zeta\left(\bar{\rho}+\delta\right)}\right)\left(1+O(\delta)\right). \label{long one}
\end{align}
Also 
\begin{equation}
\frac{1}{\zeta(2+2\delta)}=\frac{1}{\zeta(2)}+O(\delta). \label{lame help}
\end{equation}
We then divide (\ref{piece 1 done}) by $\zeta(1+\delta)$ and apply (\ref{laurent}), (\ref{long one}), and (\ref{lame help}) along with the fact that $\zeta(2)=\frac{\pi^{2}}{6}$ to write 
\begin{align}
\frac{1}{\zeta(1+\delta)}\sum_{n\leq N}\frac{\mu(n)M(n-1)}{n^{1+\delta}}\nonumber \\=\left(\sum_{\left|\gamma\right|<T}\frac{1}{\rho\left|\zeta'(\rho)\right|^{2}}-\delta\sum_{\left|\gamma\right|<T}\frac{\delta\sum_{k=2}^{\infty}\overline{\zeta^{(k)}}(\rho)\frac{\delta^{k-2}}{k!}}{\rho \left|\zeta'(\rho)\right|^{2}\zeta\left(\bar{\rho}+\delta\right)}\right)\left(1+O(\delta)\right)\nonumber \\ +O\left(\frac{\delta}{N^{\delta-\varepsilon}}\right)\left(\sum_{\left|\gamma\right|<T}\frac{1}{\rho\zeta'(\rho)}+\sum_{\left|\gamma\right|<T}\frac{1}{\rho\zeta'(\rho)}\left|\frac{1+\delta-\rho}{\varepsilon-1/2-\delta+\rho}\right|\right) \nonumber \\ +O\left(\delta^{2}\right)+O\left(\frac{\delta}{N^{1/2+\delta-\varepsilon}}\right) \nonumber \\ +O\left(\delta\right)+O\left(\frac{\delta}{N^{5/2+\delta-\varepsilon}}\right)  \nonumber \\
-\frac{3}{\pi^{2}}\left(1-\frac{1}{N^{\delta}}\right)+O\left(\delta\right)+O\left(\frac{\delta}{N^{1/2+\delta}}\right)+O\left(\frac{\delta N^{2-\delta}}{T^{1-\epsilon}}\right),\label{we there!}
\end{align}
for $\textnormal{Re}(\delta)>0$ and $N\geq \textnormal{max}\left\{N_{1}(\varepsilon),N_{2}\right\}$. Again, there is no hidden dependence on $\delta$ in (\ref{we there!})'s $O(.)$ terms. The first three terms on (\ref{we there!})'s sixth line come from (\ref{piece 1 done})'s sixth line and we thus note from Lemma \ref{elem lemma}'s result (\ref{dlmf gen err 2}) that
\begin{equation}
0\leq \left|\frac{3}{\pi^{2}}\frac{1}{N^{\delta}}+O\left(\frac{\delta}{N^{1/2+\delta}}\right)\right|<\frac{3}{\pi^{2}}. \label{dlmf gen err 3}
\end{equation}
The result on (\ref{we there!})'s fourth line is achieved from (\ref{piece 1 done})'s third line with (\ref{laurent}) and the fact that 
\begin{equation}
\left|\frac{1+\delta}{\varepsilon-1/2-\delta}\right|=O(1) \label{easy simple fact}
\end{equation}
for all $\delta\geq 0$ and $0<\varepsilon<1/2+\delta$. Meanwhile (\ref{we there!})'s fifth line is achieved from (\ref{piece 1 done})'s fourth and fifth lines with the fact that 
\begin{equation}
\sum_{k=1}^{\infty}\left|\frac{\left(2\pi i\right)^{2k}}{(2k)!k\zeta(2k+1)\zeta(2k+1+\delta)} \right|< \sum_{k=1}^{\infty}\frac{\left(2\pi\right)^{2k}}{(2k)!k\zeta(2k+1)}=O(1) \label{k sum bound}
\end{equation}
for all $\delta \geq 0$ as well as the fact that
\begin{equation}
\sum_{k=1}^{\infty}\frac{(2\pi)^{2k}}{(2k)!k\zeta(2k+1)}\left|\frac{2k+1+\delta}{\varepsilon-2k-1/2-\delta}\right|=O(1)
\end{equation}
for any $\delta\geq 0$ with $0<\varepsilon<5/2+\delta$.

We next note that, although $\sum_{\rho}\frac{1}{\rho\zeta'(\rho)}$ is not absolutely convergent, from (\ref{pre mertens big}) with $n=1$ we have the conditional convergence
\begin{equation}
\lim_{T\rightarrow \infty}\sum_{\left|\gamma\right|<T}\frac{1}{\rho\zeta'(\rho)}=\frac{5}{2}+\sum_{k=1}^{\infty}\frac{(2\pi i)^{2k}}{(2k)!k\zeta(2k+1)}=O(1). \label{thats convenient}
\end{equation}
We apply (\ref{thats convenient}) along with (\ref{write it out}) in (\ref{we there!}) to give
\begin{align}
\frac{1}{\zeta(1+\delta)}\sum_{n\leq N}\frac{\mu(n)M(n-1)}{n^{1+\delta}}\nonumber \\=\left(\sum_{0<\gamma<T}\frac{1}{\left|\rho \zeta'(\rho)\right|^{2}}-\delta\sum_{\left|\gamma\right|<T}\frac{\delta\sum_{k=2}^{\infty}\overline{\zeta^{(k)}}(\rho)\frac{\delta^{k-2}}{k!}}{\rho \left|\zeta'(\rho)\right|^{2}\zeta\left(\bar{\rho}+\delta\right)}\right)\left(1+O(\delta)\right)\nonumber \\ +O\left(\frac{\delta}{N^{\delta-\varepsilon}}\right)\left(O(1)+\sum_{\left|\gamma\right|<T}\frac{1}{\rho\zeta'(\rho)}\left|\frac{1+\delta-\rho}{\varepsilon-1/2-\delta+\rho}\right|\right) \nonumber \\ +O\left(\delta^{2}\right)+O\left(\frac{\delta}{N^{1/2+\delta-\varepsilon}}\right) \nonumber \\ +O\left(\delta\right)+O\left(\frac{\delta}{N^{5/2+\delta-\varepsilon}}\right)  \nonumber \\
-\frac{3}{\pi^{2}}\left(1-\frac{1}{N^{\delta}}\right)+O\left(\delta\right)+O\left(\frac{\delta}{N^{1/2+\delta}}\right)+O\left(\frac{\delta N^{2-\delta}}{T^{1-\epsilon}}\right),\label{we there forreal!}
\end{align}
We next set $\delta=\delta(T)$ in (\ref{we there forreal!}) with $\delta(T)$ defined by Lemma \ref{delta lemma}'s (\ref{x bound})-(\ref{sneaky}), apply (\ref{probterm1})-(\ref{sneaked}), and simplify, recalling that we chose $0<\varepsilon<1/2$, to then write
\begin{align}
\frac{1}{\zeta\left(1+\delta(T)\right)}\sum_{n\leq N}\frac{\mu(n)M(n-1)}{n^{1+\delta(T)}}=-\frac{3}{\pi^{2}}\left(1-\frac{1}{N^{\delta(T)}}\right)+\sum_{0<\gamma<T}\frac{1}{\left|\rho \zeta'(\rho)\right|^{2}} \nonumber \\ +O\left(\frac{1}{N}\right)+O\left(\frac{1}{N^{1+O(1/N)-\varepsilon}}\right)+O\left(\frac{N^{1-O(1/N)}}{T^{1-\epsilon}}\right). \label{we really there!}
\end{align}
We next note that 
\begin{equation}
\frac{N^{1-O(1/N)}}{T^{1-\epsilon}}=O\left(\frac{N}{T^{1-\epsilon}}\right) \label{doi}
\end{equation}
and we thus set 
\begin{equation}
N=T^{1-\epsilon-c} \label{nset}
\end{equation} 
for some $0<c<1-\epsilon$ so that 
\begin{equation}
\frac{N}{T^{1-\epsilon}}=O\left(\frac{1}{T^{c}}\right). \label{nnice}
\end{equation}
We then apply (\ref{doi}), (\ref{nset}), and (\ref{nnice}) in (\ref{we really there!}) and take $T\rightarrow \infty$ to give the result (\ref{main lim 1}), where we have combined the arbitrary $\epsilon$ and $c$. Since $\frac{3}{\pi^{2}N^{\delta(T)}}$ is the last surviving term from (\ref{dlmf gen err 3}) in (\ref{we really there!}) with $T\rightarrow \infty$, we may conclude that $0\leq \frac{1}{N^{\delta(T)}}<1$, which, with (\ref{nset}), completes the proof. 
\end{proof}

The proof of Corollary \ref{maincor} then follows quickly from the result (\ref{main lim 1}):

\subsubsection{Proof of Corollary \ref{maincor}}
\begin{proof}
We note that
\begin{equation}
-\frac{3}{\pi^{2}}\left(1-T^{(c-1)\delta(T)}\right)=O(1)
\end{equation}
as $T\rightarrow \infty$. We also note that the series (\ref{write it out}) is strictly positive. Therefore, if (\ref{main lim 1})'s right-hand side is negative, then 
\begin{equation}
\sum_{0<\gamma<T}\frac{1}{\left|\rho\zeta'(\rho)\right|^{2}}=O(1)
\end{equation}
and hence 
\begin{equation}
\frac{1}{\left|\rho\zeta'(\rho)\right|^{2}}=o(1) \label{obviously we have}
\end{equation}
as $\gamma \rightarrow \infty$, which immediately gives (\ref{base ub}), completing the proof. 
\end{proof}

\subsection{Proof of Theorem \ref{mainmainthm2} and Corollary \ref{maincor2}}
Similarly to the prior section, we present three lemmas that respectively describe the formula for $L(n-1)$, formulae for finite versions of (\ref{elem2}) and (\ref{dlmf gen2}), and our choice of $\delta(T)$:

\subsubsection{Lemmas}
\begin{lemma}
Under RH and SZC 
\begin{equation}
L(n-1)=\frac{n^{1/2}}{\zeta(1/2)}+\sum_{\left|\gamma\right|<T}\frac{\zeta(2\rho)n^{\rho}}{\rho\zeta'(\rho)}-\frac{\lambda(n)}{2}+O(1)+O\left(\frac{1}{n^{1/2}}\right)+O\left(\frac{n^{2}}{T^{1-\epsilon}}\right) \label{suml big}
\end{equation}
where $\epsilon>0$ is arbitrary. \label{humphlemma}
\end{lemma}
\begin{proof}
Humphries \cite{humphries} uses similar methods to those mentioned in the proof of (\ref{pre mertens big}) to prove that under RH and SZC
\begin{align}
L(n)=\frac{n^{1/2}}{\zeta(1/2)}+\sum_{\left|\gamma\right|<T}\frac{\zeta(2\rho)n^{\rho}}{\rho\zeta'(\rho)}+\frac{\lambda(n)}{2}+\frac{1}{2\pi i}\int_{\epsilon-i\infty}^{\epsilon+i\infty}\frac{\zeta(2s)n^{s}}{s\zeta(s)}ds \nonumber \\ +O\left(1+\frac{n\log n}{T}+\frac{n}{T^{1-\epsilon}\log n}\right) \label{humph1}
\end{align}
for arbitrary $0<\epsilon<1/2$. Note that 
\begin{equation}
1+\frac{n\log n}{T}+\frac{n}{T^{1-\epsilon}\log n}=O(1)+O\left(\frac{n^{2}}{T^{1-\epsilon}}\right) \label{conven}
\end{equation}
and additionally \cite{fawaz} \cite{humphries}
\begin{equation}
\frac{1}{2\pi i}\int_{\epsilon-i\infty}^{\epsilon+i\infty}\frac{\zeta(2s)n^{s}}{s\zeta(s)}ds=1+O\left(\frac{1}{n^{1/2}}\right). \label{humphconv}
\end{equation}
Combining (\ref{humph1}), (\ref{conven}), and (\ref{humphconv}) gives 
\begin{align}
L(n)=\frac{n^{1/2}}{\zeta(1/2)}+\sum_{\left|\gamma\right|<T}\frac{\zeta(2\rho)n^{\rho}}{\rho\zeta'(\rho)}+\frac{\lambda(n)}{2}+O(1)+O\left(\frac{1}{n^{1/2}}\right)+O\left(\frac{n^{2}}{T^{1-\epsilon}}\right). \label{pre suml big}
\end{align}
Then since $L(n-1)=L(n)-\lambda(n)$ we have (\ref{suml big}).
\end{proof}

\begin{lemma}
Under RH, for any $\varepsilon>0$ there exists an $N_{1}(\varepsilon)<\infty$ such that for all $N\geq N_{1}(\varepsilon)$ 
\begin{align}
\sum_{n\leq N}\frac{\lambda(n)}{n^{s}}=\frac{\zeta(2s)}{\zeta(s)}+\textnormal{Err}_{1}(s,N) \label{elem finite2}
\end{align}
with 
\begin{align}
\left|\textnormal{Err}_{1}(s,N)\right|=\left(1+\left|\frac{s}{1/2+\varepsilon-s}\right|\right)O\left(\frac{1}{N^{s-1/2-\varepsilon}}\right) \label{elem err2}
\end{align}
for $\textnormal{Re}(s)>1/2$ where the implied constant in (\ref{elem err2})'s $O(.)$ term has no dependence on $s$. Additionally for all $N\geq N_{2}$ (unconditionally)
\begin{align}
\sum_{n\leq N}\frac{\lambda^{2}(n)}{n^{s}}=\sum_{n\leq N}\frac{1}{n^{s}}=\zeta(s)+\textnormal{Err}_{2}(s,N) \label{dlmf gen finite2}
\end{align}
with
\begin{align}
\textnormal{Err}_{2}(s,N)=-\frac{1}{s-1}\frac{1}{N^{s-1}} \label{dlmf gen err2}
\end{align}
for $\textnormal{Re}(s)>1$ where $N_{2}<\infty$ and the implied constant in (\ref{dlmf gen err2})'s $O(.)$ term similarly has no dependence on $s$. For $s\in \mathbb{R}$ with $s>1$ we also have 
\begin{equation}
0\leq \left|\textnormal{Err}_{2}(s,N)\right|<\zeta(s). \label{dlmf gen err22}
\end{equation}
\label{elem lemma2}
\end{lemma}
\begin{proof}
The proof very closely follows that of Lemma \ref{elem lemma}, using (\ref{suml rh}) instead of (\ref{mertens rh}) for the proof of (\ref{elem finite2})-(\ref{elem err2}) by partial summation, and also applying partial summation for the proof of (\ref{dlmf gen finite2})-(\ref{dlmf gen err2}). Similarly to (\ref{dlmf gen err2}), (\ref{dlmf gen err22}) is clear from the fact that $\sum_{n}\frac{1}{n^{s}}$ is real, strictly positive, and convergent for $s\in \mathbb{R}$ with $s>1$. 
\end{proof}

Again, note that (\ref{elem err2})'s $\varepsilon$ is distinct from (\ref{suml big})'s $\epsilon$. 

\begin{lemma}
Let $X(T)$ be defined as in (\ref{x bound}) and
\begin{align}
S(T)=\textnormal{max}\left\{\sum_{0<\gamma<T}\left|\frac{\zeta(2\rho)}{\rho\zeta'(\rho)}\right|^{2}, \left|\sum_{\left|\gamma\right|<T}\frac{\zeta(2\rho)}{\rho\left|\zeta'(\rho)\right|^{2}}\right|, \right. \nonumber \\ \left. \sup_{0\leq x \leq \frac{X(T)}{2}}\left\{\left|\sum_{\left|\gamma\right|<T}\frac{\left|\zeta(2\rho)\right|^{2}x\sum_{k=2}^{\infty}\overline{\zeta^{(k)}}(\rho)\frac{x^{k-2}}{k!}}{\rho\left|\zeta'(\rho)\right|^{2}\zeta\left(\bar{\rho}+x\right)}\right|, \left|\sum_{\left|\gamma\right|<T}\frac{\zeta(2\rho)}{\rho\zeta'(\rho)}\left|\frac{1+x-\rho}{\varepsilon-1/2-x+\rho}\right|\right|,\right. \right. \nonumber \\ \left. \left. \left|\sum_{\left|\gamma\right|<T}\frac{\zeta(2\rho)x\sum_{k=2}^{\infty}\overline{\zeta^{(k)}}(\rho)\frac{x^{k-2}}{k!}}{\rho\left|\zeta'(\rho)\right|^{2}\zeta\left(\bar{\rho}+x\right)}\right|^{1/2}\right\} \right\}.
\label{da sup2}
\end{align}
Then define 
\begin{equation}
\delta(T)=\frac{1}{N}\textnormal{min}\left\{\frac{1}{\textnormal{max}\left\{S(T),1\right\}},X(T),\varepsilon\right\} \label{sneaky2}
\end{equation}
with $\varepsilon>0$. Then, under RH and SZC, $\delta(T)$ as defined by (\ref{x bound}), (\ref{da sup2}), and (\ref{sneaky2}) is greater than 0 for all $T<\infty$. Additionally (\ref{sneaked}) still holds along with
\begin{equation}
O\left(\delta(T)\right)\sum_{0<\gamma<T}\left|\frac{\zeta(2\rho)}{\rho\zeta'(\rho)}\right|^{2}=O\left(\frac{1}{N}\right), \label{probterm12}
\end{equation}
\begin{equation}
O\left(\delta(T)\right) \left|\sum_{\left|\gamma\right|<T}\frac{\zeta(2\rho)}{\rho\left|\zeta'(\rho)\right|^{2}}\right|=O\left(\frac{1}{N}\right), \label{probterm22}
\end{equation}
\begin{equation}
O\left(\delta(T)\right)\left|\sum_{\left|\gamma\right|<T}\frac{\left|\zeta(2\rho)\right|^{2}\delta(T)\sum_{k=2}^{\infty}\overline{\zeta^{(k)}}(\rho)\frac{\left(\delta(T)\right)^{k-2}}{k!}}{\rho\left|\zeta'(\rho)\right|^{2}\zeta\left(\bar{\rho}+\delta(T)\right)}\right|=O\left(\frac{1}{N}\right), \label{probterm32}
\end{equation}
\begin{equation}
O\left(\left(\delta(T)\right)^{2}\right)\left|\sum_{\left|\gamma\right|<T}\frac{\zeta(2\rho)\delta(T)\sum_{k=2}^{\infty}\overline{\zeta^{(k)}}(\rho)\frac{\left(\delta(T)\right)^{k-2}}{k!}}{\rho\left|\zeta'(\rho)\right|^{2}\zeta\left(\bar{\rho}+\delta(T)\right)}\right|=O\left(\frac{1}{N}\right), \label{probterm42}
\end{equation}
and 
\begin{equation}
O\left(\frac{\delta(T)}{N^{\delta(T)-\varepsilon}}\right)\left|\sum_{\left|\gamma\right|<T}\frac{\zeta(2\rho)}{\rho\zeta'(\rho)}\left|\frac{1+\delta(T)-\rho}{\varepsilon-1/2-\delta(T)+\rho}\right|\right|=O\left(\frac{1}{N^{1+\delta(T)-\varepsilon}}\right) \label{probterm52}
\end{equation}
as $N\rightarrow \infty$. 
\label{delta lemma2}
\end{lemma}
\begin{proof}
We first note that under RH $\textnormal{Re}(2\rho)=1$ for all $\rho$ and hence, by the fact that $\zeta(s)$ is analytic on $\textnormal{Re}(s)=1$ with $\textnormal{Im}(s)\neq 0$, 
\begin{equation}
\left|\zeta(2\rho)\right|<\infty \label{zeta2rhofinite}
\end{equation}
for all $\rho$. We next note that, similarly to (\ref{write it out}), 
\begin{equation}
\sum_{\left|\gamma\right|<T}\frac{\left|\zeta(2\rho)\right|^{2}}{\rho\left|\zeta'(\rho)\right|^{2}}=\sum_{0<\gamma<T}\left|\frac{\zeta(2\rho)}{\rho\zeta'(\rho)}\right|^{2}. \label{write it out2}
\end{equation}
It is then clear that 
\begin{equation}
\sum_{0<\gamma<T}\left|\frac{\zeta(2\rho)}{\rho\zeta'(\rho)}\right|^{2}<\infty \label{easy2}
\end{equation}
for all $T<\infty$ by (\ref{easy sum}) and (\ref{zeta2rhofinite}). We next note that 
\begin{equation}
\left|\sum_{\left|\gamma\right|<T}\frac{\left|\zeta(2\rho)\right|^{2}x\sum_{k=2}^{\infty}\overline{\zeta^{(k)}}(\rho)\frac{x^{k-2}}{k!}}{\rho\left|\zeta'(\rho)\right|^{2}\zeta\left(\bar{\rho}+x\right)}\right|<\infty \label{harder2}
\end{equation}
for all $0\leq x<X(T)$ and $T<\infty$ by (\ref{harder sum}) and (\ref{zeta2rhofinite}), as well as  
\begin{equation}
\left|\sum_{\left|\gamma\right|<T}\frac{\zeta(2\rho)}{\rho\zeta'(\rho)}\left|\frac{1+x-\rho}{\varepsilon-1/2-x+\rho}\right|\right|<\infty \label{clear2}
\end{equation}
for all $\varepsilon>0$, $0\leq x<\infty$, and $T<\infty$ by (\ref{clear sum}) and (\ref{zeta2rhofinite}). We additionally note that by SZC and (\ref{zeta2rhofinite})
\begin{equation}
\left|\sum_{\left|\gamma\right|<T}\frac{\zeta(2\rho)}{\rho\left|\zeta'(\rho)\right|^{2}}\right|<\infty \label{easy3}
\end{equation}
for all $T<\infty$ and, by (\ref{harder sum}) with (\ref{zeta2rhofinite}),
\begin{equation}
\left|\sum_{\left|\gamma\right|<T}\frac{\zeta(2\rho)x\sum_{k=2}^{\infty}\overline{\zeta^{(k)}}(\rho)\frac{x^{k-2}}{k!}}{\rho\left|\zeta'(\rho)\right|^{2}\zeta\left(\bar{\rho}+x\right)}\right|<\infty \label{harder3}
\end{equation}
for all $0\leq x<X(T)$ and $T<\infty$. By (\ref{easy2})-(\ref{harder3}) we have (\ref{da sup2})'s $S(T)<\infty$ for all $T<\infty$ and hence (\ref{sneaky2})'s $\delta(T)>0$ for all $T<\infty$. The results (\ref{probterm12})-(\ref{probterm52}) are clear from (\ref{da sup2}) and (\ref{sneaky2}).
\end{proof}

We next prove Theorem \ref{mainmainthm2}:

\subsubsection{Proof of Theorem \ref{mainmainthm2}}
\begin{proof}
We first apply Lemma \ref{humphlemma}'s result (\ref{suml big}) to write
\begin{align}
\sum_{n\leq N}\frac{\lambda(n)L(n-1)}{n^{1+\delta}}=\frac{1}{\zeta(1/2)}\sum_{n\leq N}\frac{\lambda(n)}{n^{1/2+\delta}}+\sum_{\left|\gamma\right|<T}\frac{\zeta(2\rho)}{\rho\zeta'(\rho)}\sum_{n\leq N}\frac{\lambda(n)}{n^{1-\rho+\delta}}-\frac{1}{2}\sum_{n\leq N}\frac{\lambda^{2}(n)}{n^{1+\delta}} \nonumber \\
+O\left(1\right)\sum_{n\leq N}\frac{\lambda(n)}{n^{1+\delta}}+O\left(\sum_{n\leq N}\frac{1}{n^{3/2+\delta}}\right)+O\left(\frac{1}{T^{1-\epsilon}}\sum_{n\leq N}n^{1-\textnormal{Re}(\delta)}\right). \label{suml raw}
\end{align}
We then apply Lemma \ref{elem lemma2}'s results (\ref{elem finite2})-(\ref{elem err2}) and (\ref{dlmf gen finite2})-(\ref{dlmf gen err2}) in (\ref{suml raw}) to write
\begin{align}
\sum_{n\leq N}\frac{\lambda(n)L(n-1)}{n^{1+\delta}}=\frac{1}{\zeta(1/2)}\left(\frac{\zeta(1+2\delta)}{\zeta(1/2+\delta)}+\left(1+\left|\frac{1/2+\delta}{\varepsilon-\delta}\right|\right)O\left(\frac{1}{N^{\textnormal{Re}(\delta)-\varepsilon}}\right)\right) \nonumber \\ +\sum_{\left|\gamma\right|<T}\frac{\zeta(2\rho)}{\rho\zeta'(\rho)}\left(\frac{\zeta(2(1-\rho)+2\delta)}{\zeta(1-\rho+\delta)}+\left(1+\left|\frac{1+\delta-\rho}{\varepsilon-1/2-\delta+\rho}\right|\right)O\left(\frac{1}{N^{\textnormal{Re}(\delta)-\varepsilon}}\right)\right) \nonumber\\ -\frac{1}{2}\left(\zeta(1+\delta)-\frac{1}{\delta N^{\textnormal{Re}(\delta)}}\right) \nonumber\\ +O(1)\left(\frac{\zeta(2+2\delta)}{\zeta(1+\delta)}+\left(1+\left|\frac{1+\delta}{\varepsilon-1/2-\delta}\right|\right)O\left(\frac{1}{N^{1/2+\textnormal{Re}(\delta)-\varepsilon}}\right)\right) \nonumber \\
+O\left(\frac{1}{N^{1/2+\textnormal{Re}(\delta)}}\right)+O\left(\frac{N^{2-\textnormal{Re}(\delta)}}{T^{1-\epsilon}}\right)
\label{suml less raw}
\end{align}
for $\textnormal{Re}(\delta)>0$ and $N\geq \textnormal{max}\left\{N_{1}(\varepsilon),N_{2}\right\}$ where we again choose $0<\varepsilon<1/2$. Note that there is no hidden dependence on $\delta$ in (\ref{suml less raw})'s $O(.)$ terms. We then consider real $\delta>0$ and divide both sides of (\ref{suml less raw}) by $\zeta(1+\delta)$. We apply (\ref{laurent}), (\ref{long one}), the fact that 
\begin{equation}
\frac{1}{\zeta(1/2+\delta)}=\frac{1}{\zeta(1/2)}+O(\delta),
\end{equation}
and that, by RH and the analyticity of $\zeta(s)$ with $\textnormal{Re}(s)=1$ and $\textnormal{Im}(s)\neq 0$,
\begin{equation}
\zeta\left(2(1-\rho)+2\delta \right)=\zeta\left(2\bar{\rho}\right)+O(\delta).  \label{zetadubrho}
\end{equation}
This gives
\begin{align}
\frac{1}{\zeta(1+\delta)}\sum_{n\leq N}\frac{\lambda(n)L(n-1)}{n^{1+\delta}} \nonumber \\ =\frac{1}{2\zeta^{2}(1/2)}+O\left(\delta\right)+\left(1+\frac{1/2+\delta}{\left|\varepsilon-\delta\right|}\right)O\left(\frac{\delta}{N^{\delta-\varepsilon}}\right) \nonumber \\ 
+\left(\sum_{\left|\gamma\right|<T}\frac{\left|\zeta(2\rho)\right|^{2}}{\rho\left|\zeta'(\rho)\right|^{2}}-\delta\sum_{\left|\gamma\right|<T}\frac{\left|\zeta(2\rho)\right|^{2}\delta\sum_{k=2}^{\infty}\overline{\zeta^{(k)}}(\rho)\frac{\delta^{k-2}}{k!}}{\rho\left|\zeta'(\rho)\right|^{2}\zeta\left(\bar{\rho}+\delta\right)}\right)\left(1+O(\delta)\right) \nonumber\\ 
+O(\delta)\left(\sum_{\left|\gamma\right|<T}\frac{\zeta(2\rho)}{\rho\left|\zeta'(\rho)\right|^{2}}-\delta\sum_{\left|\gamma\right|<T}\frac{}{}\frac{\zeta(2\rho)\delta\sum_{k=2}^{\infty}\overline{\zeta^{(k)}}(\rho)\frac{\delta^{k-2}}{k!}}{\rho\left|\zeta'(\rho)\right|^{2}\zeta\left(\bar{\rho}+\delta\right)}\right)\left(1+O(\delta)\right)
\nonumber \\
+O\left(\frac{\delta}{N^{\delta-\varepsilon}}\right)\left(\sum_{\left|\gamma\right|<T}\frac{\zeta(2\rho)}{\rho\zeta'(\rho)}+\sum_{\left|\gamma\right|<T}\frac{\zeta(2\rho)}{\rho\zeta'(\rho)}\left|\frac{1+\delta-\rho}{\varepsilon-1/2-\delta+\rho}\right|\right)
\nonumber\\ 
-\frac{1}{2}\left(1-\frac{1}{N^{\delta}}\right)+O(\delta)\nonumber\\
+O\left(\delta^{2}\right)+O\left(\frac{\delta}{N^{1/2+\delta-\varepsilon}}\right) \nonumber \\
+O\left(\frac{\delta}{N^{1/2+\delta}}\right)+O\left(\frac{\delta N^{2-\delta}}{T^{1-\epsilon}}\right) \label{suml mess}
\end{align}
Where, again, there is no hidden dependence on $\delta$ in (\ref{suml mess})'s $O(.)$ terms. Note that the terms on (\ref{suml mess})'s sixth line come from (\ref{suml less raw})'s third line and thus, from Lemma \ref{elem lemma2}'s result (\ref{dlmf gen err22}), we have
\begin{equation}
0\leq \frac{1}{2N^{\delta}}<\frac{1}{2}. \label{dlmf gen err 32}
\end{equation}
The third and fourth lines of (\ref{suml mess}) are an expanded form of
\begin{align}
\sum_{\left|\gamma\right|<T}\frac{\zeta(2\rho)}{\rho\zeta'(\rho)}\left(\left(\frac{1}{\overline{\zeta'}(\rho)}-\frac{\delta^{2}\sum_{k=2}^{\infty}\overline{\zeta^{(k)}}(\rho)\frac{\delta^{k-2}}{k!}}{\overline{\zeta'}(\rho)\zeta\left(\bar{\rho}+\delta\right)}\right)\left(1+O(\delta)\right)\right)\left(\zeta\left(2\bar{\rho}\right)+O(\delta)\right),
\end{align}
which is reached by applying (\ref{laurent}), (\ref{long one}), and (\ref{zetadubrho}) to (\ref{suml less raw})'s second line. The seventh line of (\ref{suml mess}) is gotten from (\ref{suml less raw})'s fourth line and (\ref{easy simple fact}).

We note that, although $\sum_{\rho}\frac{\zeta(2\rho)}{\rho\zeta'(\rho)}$ is not absolutely convergent, from (\ref{pre suml big}) with $n=1$ we have the conditional convergence
\begin{equation}
\lim_{T\rightarrow \infty}\sum_{\left|\gamma\right|<T}\frac{\zeta(2\rho)}{\rho\zeta'(\rho)}=O(1). \label{suml convenient}
\end{equation}
We next set $\delta=\delta(T)$ in (\ref{suml mess}) with $\delta(T)$ defined by Lemma \ref{delta lemma2}'s (\ref{da sup2})-(\ref{sneaky2}) with (\ref{x bound}). We then apply (\ref{probterm12})-(\ref{probterm52}) along with (\ref{write it out2}) and (\ref{suml convenient}) to write 
\begin{align}
\frac{1}{\zeta\left(1+\delta(T)\right)}\sum_{n\leq N}\frac{\lambda(n)L(n-1)}{n^{1+\delta(T)}} \nonumber\\ =\frac{1}{2\zeta^{2}(1/2)}+O\left(\frac{1}{N}\right)+\left(1+\frac{1/2+O(1/N)}{\varepsilon-O(1/N)}\right)O\left(\frac{1}{N^{1+O(1/N)-\varepsilon}}\right) \nonumber\\ 
+\sum_{0<\gamma<T}\left|\frac{\zeta(2\rho)}{\rho\zeta'(\rho)}\right|^{2}+O\left(\frac{1}{N}\right) \nonumber\\ +O\left(\frac{1}{N^{1+O(1/N)-\varepsilon}}\right) \nonumber\\ -\frac{1}{2}\left(1-\frac{1}{N^{\delta(T)}}\right)+O\left(\frac{1}{N}\right) \nonumber \\
+O\left(\frac{1}{N^{2}}\right) +O\left(\frac{1}{N^{3/2+O(1/N)-\varepsilon}}\right) \nonumber\\ 
+O\left(\frac{1}{N^{3/2+O(1/N)}}\right) +O\left(\frac{N^{1-O(1/N)}}{T^{1-\epsilon}}\right), \label{suml mess big os}
\end{align}
which we may simplify, recalling that we set $0<\varepsilon<1/2$, to give
\begin{align}
\frac{1}{\zeta\left(1+\delta(T)\right)}\sum_{n\leq N}\frac{\lambda(n)L(n-1)}{n^{1+\delta(T)}}=\frac{1}{2}\left(\frac{1}{\zeta^{2}(1/2)}-\left(1-\frac{1}{N^{\delta(T)}}\right)\right) \nonumber\\ +\sum_{0<\gamma<T}\left|\frac{\zeta(2\rho)}{\rho\zeta'(\rho)}\right|^{2}+O\left(\frac{1}{N}\right)+O\left(\frac{1}{N^{1+O(1/N)-\varepsilon}}\right)+O\left(\frac{N^{1-O(1/N)}}{T^{1-\epsilon}}\right). \label{we really there! 2}
\end{align} 
We then apply (\ref{doi}), (\ref{nset}), and (\ref{nnice}) in (\ref{we really there! 2}) and take $T\rightarrow \infty$ to give the result (\ref{main lim 2}), where we have combined the arbitrary $\epsilon$ and $c$. By (\ref{dlmf gen err 32}) we may conclude that $0\leq \frac{1}{N^{\delta(T)}}<1$, which, with (\ref{nset}), completes the proof.

\end{proof}

We then prove Corollary \ref{maincor2}:
\subsubsection{Proof of Corollary \ref{maincor2}}
\begin{proof}
We note that
\begin{equation}
\frac{1}{2}\left(\frac{1}{\zeta^{2}(1/2)}-1+T^{(c-1)\delta(T)}\right)=O(1)
\end{equation}
as $T\rightarrow \infty$. We also note that the series (\ref{write it out2}) is strictly positive. Therefore, if (\ref{main lim 2})'s right-hand side is negative, then 
\begin{equation}
\sum_{0<\gamma<T}\left|\frac{\zeta(2\rho)}{\rho\zeta'(\rho)}\right|^{2}=O(1)
\end{equation}
implying
\begin{equation}
\left|\frac{\zeta(2\rho)}{\rho\zeta'(\rho)}\right|^{2}=o(1) \label{obviously we have2}
\end{equation}
and hence
\begin{equation}
\frac{1}{\zeta'(\rho)}=o\left(\left|\frac{\rho}{\zeta(2\rho)}\right|\right) \label{obvobv}
\end{equation}
as $\gamma \rightarrow \infty$. We then note that under RH \cite{littlewood1line} \cite{titchmarsh} 
\begin{equation}
\frac{1}{\zeta(1+it)}=O\left(\log \log t\right) \label{1linebd}
\end{equation}
as $t\rightarrow \infty$. Combining (\ref{obvobv}) and (\ref{1linebd}) with $\rho=1/2+i\gamma$ and noting that 
\begin{equation}
\log \log \left|2\gamma\right|=\log \log \left|\gamma\right|+O\left(\frac{1}{\log \left|\gamma\right|}\right)
\end{equation}
gives (\ref{base ub 2}).
\end{proof}


\begin{thebibliography}{1}

\bibitem{anderson}Anderson, R.J. \& Stark, H.M. (1981). Oscillation theorems, in Knopp, M.I. Analytic Number Theory: Proceedings of a Conference Held at Temple University, Philadelphia, May 12-15, 1980, in: Lecture Notes in Math. vol. 899, pg. 79-106. Springer. 

\bibitem{borwein}Borwein, P., Ferguson, R., \& Mossinghoff, M.J. (2008). Sign changes in sums of the Liouville function. Math. Comp. 77 (263), 1681-1694. 

\bibitem{bsz}Bourgain, J., Sarnak, P., \& Ziegler, T. (2012). Disjointness of Moebius from horocycle flows. In: Farkas, H., Gunning, R., Knopp, M., Taylor, B. (eds) From Fourier Analysis and Number Theory to Radon Transforms and Geometry. Developments in Mathematics, vol 28. Springer, New York, NY.

\bibitem{bourgain}Bourgain, J. (2013). On the correlation of the Moebius function with rank-one systems. J. d'Analyse Math. 120, 105-130.

\bibitem{bui gonek mili}Bui, H.M., Gonek, S.M., \& Milinovich, M.B. (2015). A hybrid Euler-Hadamard product and moments of $\zeta'(\rho)$. Forum Math. 27 (3), 1799-1828. 

\bibitem{chowla}Chowla, S. (1965). The Riemann hypothesis and Hilbert's tenth problem. Mathematics and its Applications, vol. 4. Gordon and Breach Science Publishers. 

\bibitem{darty}Dartyge, C. \& Tenenbaum, G. (2005). Sums of digits of multiples of integers. Ann. Inst. Fourier (Grenoble) 55 (7), 2423-2474. 

\bibitem{fawaz}Fawaz, A.Y. (1951). The explicit formula for $L_{0}(x)$. Proc. London Math. Soc. 1 (3), 86-103. 

\bibitem{gao zhao}Gao, P. \& Zhao, L. (2023). Lower bounds for negative moments of $\zeta'(\rho)$. Mathematika 69 (4), 1081-1103. 

\bibitem{gonek1}Gonek, S.M. (1989). On negative moments of the Riemann zeta function. Mathematika 36, 71-88.

\bibitem{greentao}Green, B. \& Tao, T. (2012). The M{\"o}bius function is strongly orthogonal to nilsequences. Ann. Math. 175 (2), 541-566.

\bibitem{green}Green, B. (2012). On (not) computing the M{\"o}bius function using bounded depth circuits. Combinatorics, Probability and Computing 21 (6), 942-951.

\bibitem{hardy wright}Hardy, G.H. \& Wright, E.M. (1979). An introduction to the theory of numbers (5th ed.), Oxford. 

\bibitem{hlz}Heap, W., Li, J., \& Zhao, J. (2022). Lower bounds for discrete negative moments of the Riemann zeta function. Algebra Number Theory 16(7), 1589-1625. 

\bibitem{hejhal}Hejhal, D.A. (1987). On the distribution of $\log\left|\zeta'(1/2+it)\right|$, Number theory, trace formulas and discrete groups 343-370, Boston.

\bibitem{hko}Hughes, C.P., Keating, J.P., \& O'Connell, N. (2000). Random matrix theory and the derivative of the Riemann zeta function. Proc. R. Soc. Lond. A 456, 2611-2627.

\bibitem{hmpc}Hughes, C.P., Martin, G., \& Pearce-Crump, A. (2024). A heuristic for the discrete mean values of the derivatives of the Riemann zeta function. Integers 24. 

\bibitem{humphries}Humphries, P. (2013). The distribution of weighted sums of the Liouville function and P{\'o}lya's conjecture. J. Number Theory 133, 545-582. 

\bibitem{landau}Landau, E. (1906). {\"U}ber den Zusammenhang einiger neuer S{\"a}tze der analytischen Zahlentheorie. Wiener Sitzungsberichte, Math. Klasse 115, 589–632.

\bibitem{landaumob}Landau, E. (1924). {\"U}ber die M{\"o}biussche Funktion. Rend. Circ. Mat. Palermo 48, 277-280.

\bibitem{littlewood}Littlewood, J.E. (1912). Quelques cons{\'e}quences de l'hypoth{\'e}se qua la fonction $\zeta(s)$ n'a pas des z{\'e}ros dans le demiplan $\textnormal{Re}(s)>1/2$. C.R. Acad. Sci. Paris 154, 263-266.

\bibitem{littlewood1line}Littlewood, J.E. (1923). On the function $1/\zeta(1+ti)$. Proc. London Math. Soc. 27 (2), 349-357.

\bibitem{montmob}Maier, H. \& Montgomery, H.L. (2009). The sum of the M{\"o}bius function. Bull. Lond. Math. Soc. 41 (2), 213-226.

\bibitem{mrt}Matom{\"a}ki, K., Radziwill, M., \& Tao, T. (2015). An averaged form of Chowla's conjecture. Algebra Number Theory 9(9), 2167-2196. 

\bibitem{mr}Matom{\"a}ki, K. \& Radziwill, M. (2016). Multiplicative functions in short intervals. Ann. Math. 183(3), 1015-1056. 

\bibitem{mauduit}Mauduit, C. \& Rivat, J. (2010). On a problem posed by Gel'fond: the sum of digits of primes. Ann. Math. 171 (3), 1591-1646. 

\bibitem{mandng1}Milinovich, M.B. \& Ng, N. (2012). A note on a conjecture of Gonek. Funct. Approx. Comment. Math. 46(2), 177-187

\bibitem{montgomery}Montgomery, H.L. (1973). The pair correlation of zeros of the zeta function. Proc. Sympos. Pure Math. 24, 181-193. Providence.

\bibitem{odlyzko}Odlyzko, A.M. \& te Riele, H.J.J. (1985). Disproof of the Mertens conjecture. J. reine angew. Math. 357, 138-160.

\bibitem{rudsound}Rudnick, Z. \& Soundararajan, K. (2005). Lower bounds for moments of L-functions. Proc. Natl. Acad. Sci. 102 (19), 6837-6838.

\bibitem{sarnak}Sarnak, P. (2010). Three lectures on the M{\"o}bius function randomness and dynamics. Institute for Advanced Study, Princeton.

\bibitem{saw shu}Sawin, W. \& Shusterman, M. (2022). On the Chowla and twin primes conjectures over $\mathbb{F}_{q}[T]$. Ann. Math. 196(2), 457-506. 

\bibitem{soundmob}Soundararajan, K. (2009). Partial sums of the M{\"o}bius function. J. reine angew. Math. 631, 141-152.

\bibitem{tao}Tao, T. (2016). The logarithmically averaged Chowla and Elliott conjuectures for two-point correlations. Forum Math. Pi. 4, e8. 

\bibitem{tt}Tao, T. \& Ter{\"a}v{\"a}inen, J. (2019). The structure of logarithmically averaged correlations of multiplicative functions, with applications to the Chowla and Elliott conjectures. Duke Math. J. 168(11), 1977-2027.

\bibitem{titchmarshmob}Titchmarsh, E.C. (1927). A consequence of the Riemann hypothesis. J. Lond. Math. Soc. 2, 247-254.

\bibitem{titchmarsh}Titchmarsh, E.C. \& Heath-Brown, D.R. (Ed.) (1986). The Theory of the Riemann zeta-function (2nd ed.), Oxford. 

\end{thebibliography}
\end{document}